\documentclass[11pt]{amsart}
\usepackage[utf8]{inputenc}
\usepackage{amsfonts}
\usepackage{latexsym}
\usepackage{amsmath}
\usepackage{txfonts}
\usepackage{amsthm}
\usepackage{verbatim}
\usepackage{microtype}
\usepackage{a4}

\usepackage[T1]{fontenc}   %% for LaTeX Font Warning

\newcommand{\Gcal}{\mathcal{G}}

\newcommand{\Scal}{\mathcal{S}}
\newcommand{\Xcal}{\mathcal{X}}
\newcommand{\Ycal}{\mathcal{Y}}
\newcommand{\Zcal}{\mathcal{Z}}

\newcommand{\Field}{\ensuremath{\mathbb{K}}}
\newcommand{\cField}{\ensuremath{\ol{\mathbb{K}}}}
\newcommand{\Ring}{\ensuremath{\mathfrak{R}}}

\newcommand{\ini}{\ensuremath{\operatorname{ini}_{<}}}

\renewcommand{\subset}{\subseteq}  % Subsetsymbols

\theoremstyle{plain}% default
\newtheorem{lemma}{Lemma}
\newtheorem{thm}[lemma]{Theorem}

\newtheorem{cor}[lemma]{Corollary}

\theoremstyle{definition}
\newtheorem{defi}[lemma]{Definition}

\theoremstyle{remark}
\newtheorem{rem}[lemma]{Remark}

\newcommand{\CI}[2]{\left.#1 \Perp #2 \CIhelper}
\newcommand{\CIhelper}[1][\relax]{\def\argone{#1}\def\rrrrrrrelax{\relax}
  \ifx\argone\rrrrrrrelax\right.\else\,\middle|\,#1\right.{}\fi}

\newcommand{\ol}{\overline}

\newcommand{\Xin}{\Xcal_\text{\textup{in}}} %% the input state space
\newcommand{\Xtot}{\Xcal}                   %% the total state space

\title{Generalized Binomial Edge Ideals}

\author{Johannes Rauh}

\email{jrauh@mis.mpg.de}

\address{Max Planck Institute for Mathematics in the Sciences,
Inselstrasse 22, 
D-04103 Leipzig,Germany}

\date{\today}


\begin{document}

\begin{abstract}
  This paper studies a class of binomial ideals associated to graphs with finite vertex sets.  They generalize the
  binomial edge ideals, and they arise in the study of conditional independence ideals.  A Gröbner basis can be
  computed by studying paths in the graph.  Since these Gröbner bases are square-free, generalized binomial edge ideals
  are radical.  To find the primary decomposition a combinatorial problem involving the connected components of
  subgraphs has to be solved.  The irreducible components of the solution variety are all rational.
\end{abstract}

\maketitle

\section{Introduction}
\label{sec:introduction}

Let $\Xcal_{0}$ and $\Xin$ be finite sets, $d_{0}=|\Xcal_{0}|>1$, and denote $\Xtot=\Xcal_{0}\times\Xin$.  Let $\Field$
be a field, and consider the polynomial ring $\Ring = \Field[p_{x}:x\in\Xtot]$ with $|\Xtot|$ unknowns $p_{x}$ indexed by
$\Xtot$.  For all $i,j\in\Xcal_{0}$ and all $x,y\in\Xin$ let
\begin{equation*}
  f^{ij}_{xy} = p_{ix} p_{jy} - p_{iy} p_{jx}.
\end{equation*}
For any graph $G$ on $\Xin$ the ideal $I_{G}$ in $\Ring$ generated by the binomials $f^{ij}_{xy}$ for all
$i,j\in\Xcal_{0}$ and all edges $(x,y)$ in $G$ is called the $d_{0}$th \emph{binomial edge ideal} of $G$ over $\Field$.
This is a direct generalization of~\cite{HHHKR10:Binomial_Edge_Ideals} and~\cite{Ohtani11:Ideals_of_some_2-minors},
where the same ideals have been considered in the special case $d_{0}=2$.  For a comparison of the results of the
present paper to previous results see Remark~\ref{rem:comparison-to-HHH}.

One motivation to look at generalized binomial edge ideals comes from the study of conditional independence ideals.
Given $n+1$ random variables $X_{0},X_{1},\dots,X_{n}$, generalized binomial edge ideals correspond to a collection of
statements of the form (see~\cite{HHHKR10:Binomial_Edge_Ideals} for an explanation of the notation and further details)
\begin{equation*}
  \CI{X_{0}}{X_{R}}[X_{S}=x_{S}],
\end{equation*}
where $R\cup S=\{1,\dots,n\}$.  Such statements naturally occur in the study of robustness.  Implications of the
algebraic study of generalized binomial edge ideals will be studied in another paper~\cite{RauhAy12:Robustness_and_CI},
see also~\cite[Section~4]{HHHKR10:Binomial_Edge_Ideals}.  Generalized binomial edge ideals also cover the conditional
independence ideals associated with the intersection axiom in~\cite{Fink11:Binomial_ideal_of_intersection_axiom}.  A
different generalization of the results in~\cite{Fink11:Binomial_ideal_of_intersection_axiom} was recently studied
in~\cite{SwansonTaylor11:Minimial_Primes_of_CI_Ideals}.  The ideals $I^{\langle 1\rangle}$ defined
in~\cite{SwansonTaylor11:Minimial_Primes_of_CI_Ideals} are special cases of binomial edge ideal.

\section{The Gröbner basis}
\label{sec:grobner}

Choose a total order $>$ on $\Xin$ (e.g.~choose a bijection $\Xin\cong[N]$).  This induces a lexicographic monomial
order on~$\Ring$, also denoted by $>$, via
\begin{equation*}
  p_{ix}> p_{jy} \qquad\Longleftrightarrow\qquad
  \begin{cases}
    \text{ either } & i > j,\\
    \text{ or } & i = j \text{ and } x > y.
  \end{cases}
\end{equation*}
To construct a Gröbner basis for $I_{G}$ with respect to this order the following definitions are needed:
\begin{defi}
  \label{def:admissible-path}
  A path $\pi:x=x_0,x_1,\ldots,x_r=y$ from $x$ to $y$ in $G$ is called \emph{admissible} if
  \begin{enumerate}
  \item[(i)] $x_s\neq x_t$ for $s\neq t$, and $x < y$;
  \item[(ii)] for each $k=1,\ldots,r-1$ either $x_k<x$ or $x_k>y$;
  \item[(iii)] for any proper subset $\{y_1,\ldots,y_s\}$ of $\{x_1,\ldots,x_{r-1}\}$, the sequence $x,y_1,\ldots,y_s,y$
    is not a path.
  \end{enumerate}
  A function $\kappa: \{0,\dots,r\}\to [d]$ is called \emph{$\pi$-antitone} if it satisfies
%  $\kappa(0) > \kappa(r)$ and
  \begin{equation*}
    x_{s} < x_{t} \Longrightarrow \kappa(s) \ge \kappa(t),\text{ for all }0 \le s,t \le r.
  \end{equation*}
  $\kappa$ is \emph{strictly $\pi$-antitone} if it is $\pi$-antitone and satisfies $\kappa(0) > \kappa(r)$.
\end{defi}

The notion of $\pi$-antitonicity also applies to paths which are not necessarily admissible.  However, since admissible
paths are \emph{injective} (i.e.~they only pass at most once at each vertex), in the admissible case it is possible to
write $\kappa(\ell)$ instead of $\kappa(s)$, if $\ell=x_{s}$.

To any $x<y$, any path $\pi: x=x_0,x_1,\ldots,x_r=y$ from $x$ to $y$ and any function $\kappa:\{0,\dots,r\}\to\Xcal_{0}$
associate the monomial
\begin{equation*}
  u_{\pi}^{\kappa}= \prod_{k=1}^{r-1}p_{\kappa(k)x_k}.
\end{equation*}

\begin{thm}
  \label{thm:Gbasis}
  The set of binomials
  \begin{multline*}
    {\mathcal G}
    = \smash{\bigcup_{i<j} \,\bigg\{}
    \,u_{\pi}^{\kappa}f_{xy}^{\kappa(y)\kappa(x)}\,:\; x<y,\; \pi \text{ is an admissible path in }G\text{ from $x$ to $y$},
      \\
      \kappa\text{ is strictly $\pi$-antitone}\,\smash[t]{\bigg\}}
  \end{multline*}
  is a reduced Gröbner basis of $I_G$ with respect to the monomial order introduced above.
\end{thm}
The role of $\pi$-antitonicity is the following: In smaller monomials $\prod_{k=1}^{r} p_{i_{k}x_{k}}$, smaller indices
$i_{k}$ are associated to larger points~$x_{k}$.  Hence the initial term of
$u_{\pi}^{\kappa}f_{xy}^{\kappa(y)\kappa(x)}$ is $u_{\pi}^{\kappa}p_{\kappa(y)x}p_{\kappa(x)y}$.  This explains why in
the definition of $\mathcal{G}$ the point $x$ is associated to the index $\kappa(y)$, and vice versa.  The main idea of
the proof of Theorem~\ref{thm:Gbasis} is that reduction modulo $\mathcal{G}$ changes the association of the indices
$\{i_{k}\}$ and the points $\{x_{k}\}$ until the resulting monomial is minimal.  The following lemma is a first step:
\begin{lemma}
  \label{lem:pi-antiton}
  Let $\pi:x_{0},\dots,x_{r}$ be a path in $G$, and let $\kappa:\{0,\dots,r\}\to [d]$ be an arbitrary function.  If
  $\kappa$ is not $\pi$-antitone, then there exists $g\in\mathcal{G}$ such that $\ini(g)$ divides the initial term of
  $u_{\pi}^{\kappa} f_{xy}^{\kappa(y)\kappa(x)}$.
\end{lemma}
\begin{proof}
  Let $\tau:y_{0},\dots,y_{s}$ be a minimal subpath of $\pi$ with respect to the property that the restriction of
  $\kappa$ to $\tau$ is not $\tau$-antitone.  This means that $\kappa$ is $\tau_{0}$-antitone and $\tau_{s}$-antitone,
  where $\tau_{0}=y_{1},\dots,y_{s}$ and $\tau_{s}=y_{0},\dots,y_{s-1}$.  Assume without loss of generality that
  $y_{0}<y_{s}$, otherwise reverse $\tau$.  The minimality implies that $\kappa(y_{0})<\kappa(y_{s})$.  It follows that
  $\tau$ is admissible: By minimality, if $y_{0}<y_{k}<y_{s}$, then $\kappa(y_{k}) \ge \kappa(y_{s}) > \kappa(y_{0}) \ge
  \kappa(y_{k})$, a contradiction.  Define
  \begin{equation*}
    \tilde\kappa(k) =
    \begin{cases}
      \kappa(s), & \text{ if }k=0, \\
      \kappa(0), & \text{ if }k=s, \\
      \kappa(k), & \text{ if }0<k<s.
    \end{cases}
  \end{equation*}
  Then $\tilde\kappa$ is $\tau$-antitone, and
  $\ini(u_{\tau}^{\tilde\kappa}f_{y_{0}y_{s}}^{\tilde\kappa(y_{s})\tilde\kappa(y_{0})})$ divides
  $\ini(u_{\pi}^{\tilde\kappa}f_{xy}^{\tilde\kappa(y)\tilde\kappa(x)})$.
\end{proof}

\begin{proof}[Proof of Theorem~\ref{thm:Gbasis}]
  The proof is organized in three steps.

  \medskip
  \noindent
  {\em Step 1: ${\mathcal G}$ is a subset of $I_G$.}
  Let $\pi: x=x_0,x_1,\ldots,x_{r-1},x_r=y$ be an admissible path in~$G$.  The proof that $u_\pi^{\kappa}
  f_{xy}^{\kappa(j)\kappa(i)}$ belongs to $I_G$ is by induction on~$r$.  Clearly the assertion is true if $r = 1$, so
  assume $r > 1$.  Let $A = \{ x_k : x_k < x \}$ and $B = \{ x_\ell : x_\ell > y \}$.  Then either $A \neq \emptyset$ or
  $B \neq \emptyset$.

  Suppose $A \neq \emptyset$ and set $x_{k} = \max A$.  The two paths $\pi_1 : x_{k}, x_{k-1}, \ldots, x_1, x_0=x$
  and $\pi_2 : x_{k}, x_{k+1}, \ldots, x_{r-1}, x_r = y$ in $G$ are admissible.
  Let $\kappa_{1}$ and $\kappa_{2}$ be the restrictions of $\kappa$ to $\pi_{1}$ and $\pi_{2}$.  Let $a=\kappa(r)$,
  $b=\kappa(0)$ and $c=\kappa(k)$.  The calculation
  \begin{multline*}
    (p_{by}p_{ax}-p_{bx}p_{ay})p_{cx_{k}}
    \\
    =
    (p_{bx_{k}}p_{cx} - p_{bx}p_{cx_{k}})p_{ay}
    + (p_{ax_{k}}p_{by} - p_{ay}p_{bx_{k}})p_{cx}
    - (p_{ax_{k}}p_{cx} - p_{ax}p_{cx_{k}})p_{by}
  \end{multline*}
  implies that $u_{\pi}^{\kappa}f_{xy}^{ab}$ lies in the ideal generated by $u_{\pi_{1}}^{\kappa_{1}}f_{x_{k}x}^{bc}$,
  $u_{\pi_{2}}^{\kappa_{2}}f_{x_{k}y}^{ab}$ and $u_{\pi_{1}}^{\kappa_{1}}f_{x_{k}x}^{ac}$.  By induction it lies in
  $I_{G}$.
  The case $B \neq \emptyset$ can be treated similarly.

  \medskip
  \noindent
  {\em Step 2: ${\mathcal G}$ is a Gröbner basis of $I_G$.}
  Let $\pi:x_0,\dots,x_{r}$ and $\sigma:y_{0},\dots,y_{s}$ be admissible paths in $G$ with $x_{0}<x_{r}$ and
  $y_{0}<y_{s}$, and let $\kappa$ and $\mu$ be $\pi$- and $\sigma$-antitone.  By Buchberger's criterion it suffices to
  show that the $S$-pairs $S := S(u_\pi^{\kappa} f_{x_{0}x_{r}}^{\kappa(r)\kappa(0)}, u_\sigma^{\mu}
  f_{y_{0}y_{s}}^{\mu(s)\mu(0})$ reduces to zero.

  If $S\neq 0$, then $S$ is a binomial.  Write $S=S_{1}-S_{2}$, where $S_{1}=\ini(S)$.
  % Then $S_{1}$ and $S_{2}$ are two monomials in which the same lower indices occur.
  $S$ is homogeneous with respect to the multidegrees given by
  \begin{equation*}
    \deg(p_{ix})_{j} = \delta_{ij} =
    \begin{cases}
      1, & \text{ if } i=j, \\
      0, & \text{ else,}
    \end{cases}
  \end{equation*}
  and
  \begin{equation*}
    \deg(p_{ix})_{y} = \delta_{xy} =
    \begin{cases}
      1, & \text{ if } x=y, \\
      0, & \text{ else}
    \end{cases}
  \end{equation*}
  (this is a multidegree with $|\Xcal_{0}| + |\Xin|$ components).

  If $\pi$ and $\sigma$ are disjoint paths, then $S$ trivially reduces to zero, since
  $u_\pi^{\kappa}f_{x_{0}x_{r}}^{\kappa(r)\kappa(0)}$ and $u_\sigma^{\mu} f_{y_{0}y_{s}}^{\mu(s)\mu(0)}$ contain
  different variables.
  So assume that $\pi$ and $\sigma$ meet and that $S\neq 0$.  Then $S_{1}$ and $S_{2}$ are monomials, and the unknowns
  $p_{ix}$ occurring in $S_{1}$ and $S_{2}$ satisfy $x\in\pi\cup\sigma$.
  Assume that there are $x<y$ such that $D_{x}:=\min\{i\in\Xcal_{0}: p_{ix}\,|\, S_{1}\} < \max\{i\in\Xcal_{0}:
  p_{iy}\,|\,S_{1}\}=:D_{y}$.  Since $\pi\cup\sigma$ is connected there is an injective path $\tau:z_{0},\dots,z_{s}$
  from $x=z_{0}$ to $y=z_{s}$ in $\pi\cup\sigma$.  Choose a map $\lambda:\{0,\dots,s\}\to\Xcal_{0}$ such that
  $\lambda(0) = D_{x}$, $\lambda(s) = D_{y}$ and $p_{\lambda(a)a}\,|\,S_{1}$ for all $0\le a\le s$.  Then
  $u_{\tau}^{\lambda}$ divides $S_{1}$, and $\lambda$ is not $\tau$-antitone.  So Lemma~\ref{lem:pi-antiton} applies,
  and $S$ can be reduced to a smaller binomial.

  Let $S'$ be the reduction of $S$ modulo $\Gcal$.  If $S'\neq 0$, then let $S'_{1}=\ini(S')$.  The above argument shows
  that $\min\{i\in\Xcal_{0}: p_{ix}\,|\, S'_{1}\} \ge \max\{i\in\Xcal_{0}: p_{iy}\,|\,S'_{1}\}$ for all $x<y$.  This
  property characterizes $S'_{1}$ as the unique minimal monomial in $\Ring$ with multidegree $\deg(S'_{1}) = \deg(S)$.  But
% Let $S'_{2} = S' - S'_{1}$.
  since the reduction algorithm turns binomials into binomials, $S'-S'_{1}$ is also a monomial of multidegree $\deg(S)$,
  and smaller than $\deg(S'_{1})$.  This contradiction shows $S'=0$.

  \medskip
  \noindent
  {\em Step 3: $\Gcal$ is reduced.}
  Let $\pi:x_0,\dots,x_{r}$ and $\sigma:y_{0},\dots,y_{s}$ be admissible paths in $G$ with $x_{0}<x_{r}$ and
  $y_{0}<y_{s}$, and let $\kappa$ and $\mu$ be $\pi$- and $\sigma$-antitone.
  Suppose that $u_\pi^{\kappa} p_{\kappa(r)x_{0}}p_{\kappa(0)x_{r}}$ divides either $u_\sigma^{\mu} p_{\mu(s)y_{0}}p_{\mu(0)y_{s}}$ or
  $u_\sigma^{\mu} p_{\mu(s)y_{s}} p_{\mu(0)y_{0}}$.  Then $\{ x_0, \ldots, x_r \}$ is a % proper
  subset of $\{ y_0, \ldots, y_s \}$, and $\kappa(b) = \mu(\sigma^{-1}(x_{b}))$ for $0<b<r$.
  From admissibility follows $x_{0}\le y_{0}<y_{s}\le x_{r}$ and $\kappa(0)\ge\mu(0) > \mu(s)\ge\kappa(r)$.

  If $x_0<y_{0}$, then $p_{\kappa(r)x_{0}}$ divides $u_{\sigma}^{\mu}$, and so $x_{0}=y_{t}$ for some $t<s$ with
  $\mu(t)=u=\kappa(r)$.  On the other hand, since $y_{t}\le y_{0}$, it follows that $\mu(t)\ge\mu(0) > \kappa(r)$, a
  contradiction.  Hence $x_{0}=y_{0}$.  Similarly, by a symmetric argument, $x_{r}=y_{s}$.
  This means that $\pi$ is a sub-path of $\sigma$.  By Definition~\ref{def:admissible-path}, $\pi$ equals $\sigma$.
% (up to a possible change of direction).
  Therefore, $u_\pi^{\kappa} f_{x_{0}x_{r}}^{\kappa(r)\kappa(0)}$ and $u_\sigma^{\mu} f_{y_{0}y_{s}}^{\mu(s)\mu(0)}$ have the
  same (total) degree, and hence they agree.
\end{proof}

\begin{cor}
  \label{cor:radical}
  $I_G$ is a radical ideal.
\end{cor}

\begin{proof}
  The assertion follows from Theorem~\ref{thm:Gbasis} and the following general fact: A homogeneous ideal that has a
  Gröbner basis with square-free initial terms is radical.  See the proof
  of~\cite[Corollary~2.2]{HHHKR10:Binomial_Edge_Ideals} for details.
\end{proof}

\section{The primary decomposition}
\label{sec:primdec}

Since $I_{G}$ is radical, in order to compute the primary decomposition of the ideal it is enough to compute the minimal
primes.  From this it will be easy to deduce the irreducible decomposition of the variety $V_{G}$ of $I_{G}$ in the case
of characteristic zero.
The following definition is needed: Two vectors $v,w$ (living in the same $\Field$-vector space) are \emph{proportional}
whenever $v=\lambda w$ or $w=\lambda v$ for some $\lambda\in\Field$.  A set of vectors is \emph{proportional} if each pair
is proportional.  Since $\lambda=0$ is allowed, proportionality is not transitive: If $v$ and $w$ are proportional and
if $u$ and $v$ are proportional, then % we can conclude that
$u$ and $w$ need not proportional, because $v$ may vanish.

Let $V_{G}$ be the variety of $I_{G}$, which is a subset of $\Field^{\Xcal_{0}\times\Xin}$.  As usual, elements of
$\Field^{\Xcal_{0}\times\Xin}$ will be denoted with the same symbol $p=(p_{ix})_{i\in\Xcal_{0},x\in\Xin}$ as the
unknowns in the polynomial ring $\Ring = \Field[p_{ix}:(i,x)\in\Xcal_{0}\times\Xin]$.  Any
$p\in\Field^{\Xcal_{0}\times\Xin}$ can be written as a $d_{0}\times|\Xin|$-matrix.  Each binomial equation in $I_{G}$
imposes conditions on this matrix saying that certain submatrices have rank 1.  For a fixed edge $(x,y)$ in $G$ the
equations $f^{ij}_{xy}=0$ for all $i,j\in\Xcal_{0}$ require that the submatrix $(p_{kz})_{k\in\Xcal_{0},z\in\{x,y\}}$
has rank one.  More generally, if $K\subset G$ is a clique (i.e.~a complete subgraph), then the submatrix
$(p_{kz})_{k\in\Xcal_{0},z\in K}$ has rank one.  This means that all columns of this submatrix are proportional.  The
columns of $p$ will be denoted by $\tilde p_{x}$, $x\in\Xin$.  A point $p$ lies in $V_{G}$ if and only if $\tilde p_{x}$
and $\tilde p_{y}$ are proportional for all edges $(x,y)$ of $G$.

Even if the graph $G$ is connected, not all columns $\tilde p_{x}$ must be proportional to each other, since
proportionality is not a transitive relation.  Instead, there are ``blocks'' of columns such that all columns within one
block are proportional.
For any subset $\Ycal\subseteq\Xin$ denote by $G_{\Ycal}$ the subgraph of $G$ induced by $\Ycal$.
Then:
\begin{itemize}
\item A point $p$ lies in $V_{G}$ if and only if $\tilde p_{x}$ and $\tilde p_{y}$ are proportional whenever
  $x,y\in\Scal$ lie in the same connected component of $G_{\Scal}$, where $\Scal=\{x\in\Xin:\tilde p_{x}\neq 0\}$.
\end{itemize}

Let $V_{G,\Ycal}$ be the set of all $p\in\Field^{\Xcal_{0}\times\Xin}$ for which
$\tilde p_{x}=0$ for all $x\in\Xin\setminus\Ycal$ and for which $\tilde p_{x}$ and $\tilde p_{y}$ are
proportional whenever $x,y\in\Xin$ lie in the same connected component of $G_{\Ycal}$.  Then
\begin{equation}
  \label{eq:decomposition-VG}
  V_{G} = \bigcup_{\Ycal\subseteq\Xin} V_{G,\Ycal}.
\end{equation}
The sets $V_{G,\Ycal}$ are rational irreducible algebraic varieties:
\begin{lemma}
  \label{lem:VGY-is-sirreducible}
  For any $\Ycal\subseteq\Xin$ the set $V_{G,\Ycal}$ is the variety of the ideal $I_{G,\Ycal}$
  generated by the monomials
  \begin{equation}
    \label{eq:monomials}
    p_{ix} \qquad\text{for all }x\in\Xin\setminus\Ycal\text{ and }i\in\Xcal_{0},
  \end{equation}
  and the binomials $f_{xy}^{ij}$ for all $i,j\in\Xcal_{0}$ and all $x,y\in\Ycal$ that lie in the same connected
  component of $G_{\Ycal}$.  The ideal $I_{G,\Ycal}$ is prime, and the variety $V_{G,\Ycal}$ is rational.
\end{lemma}
\begin{proof}
%  The defining equations of $I_{G,\Ycal}$.
  The first statement follows from the definition of $V_{G,\Ycal}$.  Write $I^{1}_{G,\Ycal}$ for the ideal generated by
  all monomials \eqref{eq:monomials}, and for any $\Zcal\subseteq\Ycal$ write $I^{2}_{\Zcal}$ for the ideal generated by
  the binomials $f_{xy}^{ij}$, with $i,j\in\Xcal_{0}$ and $x,y\in\Zcal$.  Then $I^{1}_{G,\Ycal}$ is obviously prime.
  Each of the $I^{2}_{\Zcal}$ is a $2\times2$ determinantal ideal.  It is a classical (but difficult) result that this
  ideal is the defining ideal of a Segre embedding, and that it is prime (see~\cite{Sturmfels91:GB_of_toric_varieties}
  for a rather modern proof).  In fact, both $I^{1}_{G,\Ycal}$ and $I^{2}_{\Zcal}$ are geometrically prime, i.e.~they
  remain prime over any field extension.  Hence the ideal $I_{G,\Ycal}$ is the sum of the geometrically prime ideals
  $I^{1}_{G,\Ycal}$ and $I^{2}_{\Zcal}$ for all connected components $\Zcal$ of $G_{\Ycal}$, and since the defining
  equations of all these ideals involve disjoint sets of unknowns, $I_{G,\Ycal}$ itself is prime.  $V_{G,\Ycal}$ is
  rational, since the varieties of $I^{1}_{G,\Ycal}$ and $I^{2}_{\Zcal}$ are rational.
\end{proof}
The decomposition~\eqref{eq:decomposition-VG} is not the irreducible decomposition of $V_{G}$, because the union is
redundant.  The redundant components can be removed using Lemma~\ref{lem:VGY-is-sirreducible}:
\begin{lemma}
  \label{lem:VGYcontainsVGZ}
  Let $\Ycal,\Zcal\subseteq\Xin$.  Then $V_{G,\Ycal}$ is contained in $V_{G,\Zcal}$ if and only
  if the following two conditions are satisfied:
  \begin{itemize}
  \item $\Ycal\subseteq\Zcal$.
  \item If $x,y\in\Ycal$ are connected in $G_{\Zcal}$, then they are connected in $G_{\Ycal}$.
  \end{itemize}
\end{lemma}
\begin{proof}
  Assume that $V_{G,\Ycal}\subseteq V_{G,\Zcal}$.  Then $I_{G,\Ycal}\supseteq I_{G,\Zcal}$.  For
  any $x\in\Xin\setminus\Zcal$ and any $i\in\Xcal_{0}$ this implies $p_{ix}\in I_{G,\Ycal}$.  On the
  other hand, Lemma~\ref{lem:VGY-is-sirreducible} shows that the point with coordinates
  \begin{equation*}
    p_{iy}=
    \begin{cases}
      1, & \qquad\text{if }y\in\Ycal,\\
      0, & \qquad\text{else},
    \end{cases}
  \end{equation*}
  lies in $V_{G,\Ycal}$ and hence in $V_{G,\Zcal}$.  This implies $x\in\Xin\setminus\Ycal$; and so
  $\Ycal\subseteq\Zcal$.

  Let $x\in\Ycal$.
  Choose two linearly independent non-zero vectors $v,w\in\Field^{d_{0}}$.  By Lemma~\ref{lem:VGY-is-sirreducible} the
  matrix with columns
  \begin{equation*}
    \tilde p_{y}=
    \begin{cases}
      v, & \qquad\text{if }y\in\Ycal\text{ is connected to }x\text{ in }G_{\Ycal},\\
      w, & \qquad\text{if }y\in\Ycal\text{ is not connected to }x\text{ in }G_{\Ycal},\\
      0, & \qquad\text{else},
    \end{cases}
  \end{equation*}
  is contained in $V_{G,\Ycal}$ and hence in $V_{G,\Zcal}$.  Therefore, if $z$ is connected to $x$ in
  $G_{\Zcal}$, then it is connected to $x$ in $G_{\Ycal}$.

  Conversely, if both conditions are satisfied, then all defining equations of $I_{G,\Zcal}$ lie in~$I_{G,\Ycal}$.
\end{proof}
\begin{thm}
  \label{thm:primary-decomposition}
  The primary decomposition of $V_{G}$ is
  % \begin{equation*}
  $I_{G} = \bigcap_{\Ycal} I_{G,\Ycal}$,
  % \end{equation*}
  where the intersection is over all $\Ycal\subseteq\Xin$ such that the following holds: For any
  $x\in\Xin\setminus\Ycal$ there are edges $(x,y)$, $(x,z)$ in $G$ such that $y,z\in\Ycal$ are not
  connected in $G_{\Ycal}$.  Equivalently, for any $x\in\Xin\setminus\Ycal$ the induced subgraph
  $G_{\Ycal\cup\{x\}}$ has fewer connected components than $G_{\Ycal}$.
\end{thm}
\begin{proof}
  First, assume that $\Field$ is algebraically closed.  By~\eqref{eq:decomposition-VG} and
  Lemma~\ref{lem:VGY-is-sirreducible} it suffices to show that the condition on $\Ycal$ stated in the theorem
  characterizes the maximal sets $V_{G,\Ycal}$ in the union~\eqref{eq:decomposition-VG} (with respect to inclusion).
  This follows from Lemma~\ref{lem:VGYcontainsVGZ}.

  If $\Field$ is not algebraically closed, then one can argue as follows: By~\cite{EisenbudSturmfels96:Binomial_Ideals} a
  binomial ideal has a binomial primary decomposition over some algebraic extension field
  $\hat\Field=\Field[\alpha_{1},\dots,\alpha_{k}]$.  The algebraic numbers $\alpha_{1},\dots,\alpha_{k}$ are coefficients of
  the defining equations of the primary components.  Let $\cField$ be the algebraic closure of $\Field$.  Since the ideals
  $I_{G,\Ycal}$ are defined by pure differences and since the ideals $\cField\otimes I_{G,\Ycal}$ are the primary components
  of $\cField\otimes I_{G,\Ycal}$ in $\cField\otimes\Ring$ it follows that the ideals $I_{G,\Ycal}$ are already the primary
  components of $I_{G}$ (in other words, the primary decomposition is independent of the base field).
\end{proof}

\begin{rem}[Comparison to~\cite{Ohtani11:Ideals_of_some_2-minors,HHHKR10:Binomial_Edge_Ideals}]
  \label{rem:comparison-to-HHH}
  Both~\cite{Ohtani11:Ideals_of_some_2-minors} and~\cite{HHHKR10:Binomial_Edge_Ideals} discuss Gröbner bases and primary
  decompositions of binomial edge ideals with $d_{0}=2$.  % The Gröbner basis of
  Theorem~\ref{thm:Gbasis} generalizes Theorems~2.1 from~\cite{HHHKR10:Binomial_Edge_Ideals} and Theorem~3.2
  in~\cite{Ohtani11:Ideals_of_some_2-minors}.  While the proofs in~\cite{HHHKR10:Binomial_Edge_Ideals}
  and~\cite{Ohtani11:Ideals_of_some_2-minors} use a case by case analysis, the proof of Theorem~\ref{thm:Gbasis} is more
  conceptual.

  The primary decomposition in Theorem~\ref{thm:primary-decomposition} generalizes Theorem~3.2
  from~\cite{HHHKR10:Binomial_Edge_Ideals}.  The proof of Theorem~\ref{thm:primary-decomposition} relied on the
  irreducible decomposition of the corresponding variety, while the proof
  in~\cite{HHHKR10:Binomial_Edge_Ideals} directly shows the equality of the two ideals.

  Instead of describing the primary decomposition explicitly, \cite{Ohtani11:Ideals_of_some_2-minors} presents an
  algorithm to compute the primary decomposition.  Since the primary decomposition of a binomial edge ideal is
  independent of~$d_{0}$, the same algorithm applies for all~$d_{0}$.  A nice feature of the algorithm is that it works
  graph-theoretically.
\end{rem}

\subsection*{Acknowledgement}

This work has been supported by the Volkswagen Foundation.  I thank Nihat Ay, Thomas Kahle, Seth Sullivant, Jürgen
Herzog and Takayuki Hibi.  Further thanks goes to the reviewer for many helpful remarks.

\bibliographystyle{IEEEtranSpers}
\bibliography{general}

\end{document}